\documentclass[fleqn,10pt]{wlscirep}
\usepackage{graphicx}

\usepackage{dcolumn}
\usepackage{bm}

\usepackage[utf8]{inputenc}
\usepackage[T1]{fontenc}
\usepackage{mathptmx}
\usepackage{etoolbox}
\usepackage{subfigure}
\usepackage{mathrsfs}
\usepackage{dsfont}

\usepackage{tabularx}
\usepackage{xcolor}
\usepackage{fancyhdr}
\usepackage{amsfonts}
\usepackage{amsmath}
\usepackage{enumitem}
\usepackage{xspace}
\usepackage{xparse}
\usepackage{booktabs}
\usepackage{multirow}
\usepackage{float}

\newtheorem{lemma}{Lemma}

\newtheorem{theorem}{Theorem}
\usepackage{threeparttable}

\DeclareDocumentCommand\newstep{o}{%
\item\IfNoValueTF{#1}{}{#1 \textendash\xspace}}
\newlist{steps}{enumerate}{1}
\setlist[steps]{label=\textit{Step \arabic*:},leftmargin=*}

\title{Total value adjustment of Bermudan option valuation \\under pure jump L\'evy fluctuations}

\author[1]{Gangnan Yuan}
\author[1]{Ding Deng}
\author[2]{Jinqiao Duan}
\author[1]{Weiguo Lu}
\author[3,1,*]{Fengyan Wu}

\affil[1]{Department of Mathematics, University of Macau, Macao, China}
\affil[2]{Department of Applied Mathematics, Illinois Institute of Technology, Chicago, Illinois 60616, USA}
\affil[3]{College of Mathematics and Statistics, Chongqing University, Chongqing 401331, China}
\affil[*]{The corresponding author: fengyanwu@um.edu.mo}

\begin{abstract}
During the COVID-19 pandemic, many institutions have announced that their counterparties are struggling to fulfill contracts. Therefore, it is necessary to consider the counterparty default risk when pricing options. After the 2008 financial crisis, a variety of value adjustments have been emphasized in the financial industry. The total value adjustment (XVA) is the sum of multiple value adjustments, which is also investigated in many stochastic models such as Heston \cite{salvador2020total}
and Bates \cite{goudenege2020computing} models. In this work, a widely used pure jump L\'evy process, the CGMY process has been considered for pricing a Bermudan option with various value adjustments. Under a pure jump L\'evy process, the value of derivatives satisfies a fractional partial differential equation (FPDE). Therefore, we construct a method which combines Monte Carlo with finite difference of FPDE (MC-FF) to find the numerical approximation of exposure, and compare it with the benchmark Monte Carlo-COS (MC-COS) method. We use the discrete energy estimate method, which is different with the existing works, to derive the convergence of the numerical scheme. Based on the numerical results, the XVA is computed by the financial exposure of the derivative value.
\end{abstract}

\begin{document}
\flushbottom
\maketitle

\section*{Introduction}
Before 2007, investors believed that large financial institutions such as banks would not have the risk of default or bankruptcy, which was a very popular concept of "too big to fail". As we already know, Long Term Capital Management, Enron, Lehman Brothers and many other large financial institutions went bankrupt during the last finance crisis. In recent years, especially since the COVID-19 pandemic, the significance of the counterparty credit risk (CCR) has become increasingly prominent. In particular, institutions have gradually strengthened the risk management of financial derivatives traded over the counter (OTC). At the same time, industry regulations such as Basel regulations and the International Financial Reporting Standards (IFRS) also required the financial institutions to charge the counterparty a premium which is to balance the credit risk. This premium is usually called credit valuation adjustment (CVA) and Gregory \cite{gregory2012counterparty} gave a very detailed introduction of CCR and CVA. In recent years, many other valuation adjustments have also been discussed as a supplement of CVA.  The most significant ones are funding value adjustment (FVA) and capital value adjustment (KVA). Together with CVA, these three adjustments form the so-called XVA, i.e., \\ \centerline{\textit{XVA=CVA+FVA+KVA.}} Although there are some other adjustments, such as debit value adjustment (DVA) and market value adjustment (MVA), their existences and effects have some debates; see Kenyon\cite{kenyon2010completing} and Gregory\cite{gregory2015xva}.
The most commonly used adjustments in the industry are the three items we mentioned in the above equation. Among them, KVA is usually ignored and only considered in a few cases \cite{ruiz2015complete}, such as pricing a  10-year swap \cite{green2014kva}.

To date, there are two main approaches for the calculation of XVA. The first approach is constructed by Burgard\cite{burgard2011partial}, and they built a portfolio that contains all the underlying risk factors. Specifically, it includes defaultable bonds of each counterparty. Then they derive a partial differential equation (PDE) representation for the value of financial derivatives with CVA. By a similar method,  Arregui\cite{arregui2017pde} expanded it to the XVA and gives a lot of numerical examples, but it is still based on the classical Black-Scholes model\cite{black2019pricing}. Borovykh\cite{borovykh2018efficient} applied the method to a jump-diffusion model. As the model becomes more complex, they get a partial integro-differential equation (PIDE) representation instead of a PDE.  Salvador~\cite{salvador2020total} transferred the same method to a stochastic volatility model.

The above approach gives a PDE or PIDE to represent the value of a derivative with CVA or XVA. However, the assumptions of the above method are idealized. It requires us to find the corresponding defaultale bonds of each counterparty, which is actually hard to achieve in reality. In the industry, the more common approach to getting XVA is based on  Gregory~\cite{gregory2015xva}'s method, which is also the approach we use in this work. Since the exposure can be seen as a potential loss\cite{gregory2012counterparty}, we calculate the exposure and expected exposure (EE) related to the value of derivatives, then each component of XVA is obtained according to the definition (see Section II. A). In light of  Ruiz~\cite{ruiz2015complete}'s review article, the largest proportion of the total value adjustment is CVA, followed by FVA, and the influence of KVA can be ignored in most cases. Therefore, when we deal with a Bermudan option in this work, we will focus on CVA and FVA.\\
Based on the numerical estimation of EE, de Graaf~\cite{de2016efficient} investigated the CVA and its sensitivities under the Heston model. The Goudenege~\cite{goudenege2020computing} proposed a Hybrid Tree-Finite Difference method to compute the CVA under the Bates model. Similar to the steps of these two papers, we compute the value adjustments of a Bermudan option under a CGMY process. Comparing to the normal diffusion process, the pure jump process has better performance to capture all the empirical stylized regularities of stock price movements\cite{elliott2006option,geman2002pure}. As a tempered process, the CGMY model can be transformed into other pure jump processes, such as the Variance Gamma (VG) model\cite{oosterlee2019mathematical} and the KoBoL model\cite{cartea2007fractional}, after adjusting the parameters. This characteristic makes it a representative and widely used pure jump L\'evy model in finance. The difference between our work and  Goudenege\cite{goudenege2020computing} and de Graaf\cite{de2016efficient} is that we construct a complex pure jump model. In addition to CVA, we also consider the impact of FVA on the total value adjustments. \\
Since the exposure is based on the value of derivatives, we generate a sufficient number of paths of underlying price based on a pure jump L\'evy process by Monte Carlo simulation, and then combine them with two numerical methods to get an approximation of the option value. First, we combine Monte Carlo with finite difference of FPDE (MC-FF). Specifically, we employ the finite difference method for solving the FPDE related to the pure jump CGMY process. We use the second-order central difference operator to approximate the first-order space derivative, we utilise the tempered and weighted and shifted Gr\"{u}nwald difference (tempered-WSGD) operators to discrete the tempered fractional order derivatives \cite{lican2015}, and we employ the Crank-Nicolson scheme to discrete the time derivative. For convergence analysis, different with Li~\cite{lican2015}, the discrete energy estimate method is utilised to analyse the convergence of the numerical scheme. Second, we combine Monte Carlo with COS (MC-COS), which is a benchmark method we use to compare with the MC-FF method. This method was proposed by Fang\cite{fang2009novel}, which has a high accuracy but a long calculation time. Then, the CVA and FVA are calculated and compared by the two methods.\\
The rest of the paper is structured as follows. Section II gives the mathematical model including the reviews of XVA components and the valuation parts. In Section III, we propose two numerical methods (MC-FF, MC-COS) to compute the expected exposure of the Bermudan option under the CGMY process. And the numerical results on exposures and XVA are presented in Section IV. Finally, the conclusion and further applications of this work are drawn in Section V.

\section*{Mathematical Models For XVA and Exposure}

As we mentioned before, XVA consists of different value adjustments. In this section, we introduce the components of XVA and the exposures of the Bermudan option under a pure jump L\'evy process.
\subsection*{Components of XVA}
Traditional derivatives valuation, especially for option pricing, only considered the impact of cash flow. For simple derivative types, the pricing problem was usually just a matter of applying the correct discount factor. The global financial crisis led to a series of valuation adjustments, by considering credit risk, funding costs and regulation capital costs, to convert simple valuation into correct one. A general and simple representation \cite{gregory2015xva} of value adjustments is:
\\ \centerline{Actual Value= Base Value + XVA,}
where XVA is composed of the various value adjustments we mentioned in last section and base value refers to the original derivative value. Notice that this expression assumes the XVA components is totally separate from the actual value. It is not completely true in reality but it can almost always be considered to be a reasonable assumption in practice. \\
Before giving the definition of each component of the XVA, we first introduce several factors that can affect value adjustments:
\begin{itemize}
\item \textit{Loss given default (LGD).} The LGD refers to the proportion that would be lost if a counterparty default. It is sometimes defined as one minus the recovery rate. Although there are some works discussing the recovery rate such as Unal \cite{unal2003pricing} and Schlafer\cite{schlafer2014recovery}, the LGD is often assumed to be a constant \cite{gregory2012counterparty}. Throughout this paper, we also assume the LGD to be a constant.
\item \textit{Expected exposure (EE).} The fact of credit exposure is a  positive value of a financial derivative, and the expected exposure can be seen as the future value of the derivative. Following Gregory \cite{gregory2012counterparty}, let $V(t)$ be the value of a portfolio at time $t$, then the exposure $E(t)$ is defined by
\begin{equation} \label{Et}
    E(t)=V(t)^+,
\end{equation}
where $x^+=\textrm{max}[x,0],$
and the present expected exposure at a future time $t$ is defined by
\begin{equation}
    EE(t)=\mathbb{E}[E(t)|\mathfrak{F}_{0}],
\end{equation}
where $\mathfrak{F_{0}}$ is the filtration at time $t=0$. And we also use potential future exposure (PFE) to represent the best or worst case the buyer may face in the future. It is defined as
\begin{equation}
    PFE_{\alpha}(t)=inf\left\{ x:\mathbb{P}(E(t)\leq x)\geq\alpha\right\},
\end{equation}
where $\alpha$ takes 97.5\% and 2.5\% in this work.
\item \textit{Default probability (PD).} As its name implies, the PD is the probability of counterparty defaults, which is usually derived from credit spreads observed in the market \cite{gregory2015xva}. Let us define the default probability between two sequential times $t_{m}$ and $t_{m+1}$ is $PD(t_{m},t_{m+1})$, a commonly used approximation \cite{gregory2015xva,hull2019options} of PD is:
\begin{eqnarray}
    PD(t_{m},t_{m+1})\approx \textrm{exp}\left ( -\frac{s(t_{m})\cdot t_{m}}{LGD} \right )-\textrm{exp}\left ( -\frac{s(t_{m+1})\cdot t_{m+1}}{LGD} \right ),
\end{eqnarray}
where $s(t)$ is the credit spread at time $t$.
\end{itemize}
Recall that the XVA mainly consists of three parts: CVA, FVA and KVA.\\
\textbf{CVA} The credit value adjustment is the most important part of XVA and it is also the key expression for describing the counterparty risk. Different with traditional credit limits, the CVA can be seen as the actual price of counterparty credit risk. According to Gregory \cite{gregory2012counterparty}, assuming independence  between PD, exposure and recovery rate, a practical CVA expression is given by
\begin{equation}
\begin{aligned}
    CVA= -LGD\int_{0}^{T}EE^{*}(t)dPD(t)\approx -(1-R)\sum_{m=1}^{M}EE^{*}(t_{m})PD(t_{m-1},t_{m}),
    \end{aligned}
\end{equation}
where $R$ is the recovery rate, $EE^{*}$ is the discounted expected exposure and $\left \{  0=t_{1}<t_{2}<...<t_{M}=T \right \}  $ is a fixed time grid.\\
\\
\textbf{FVA} Limited to liquidity and capacity, a large proportion of OTC derivatives are traded without collateral. These uncollateralised trades are source of funding risk and the FVA can be broadly considered as a funding cost for these trades \cite{gregory2015xva,del2016efficient}. An intuitive FVA formula is:
\begin{equation}
\begin{aligned}
  FVA&=-\int_{0}^{T} (EPE^{*}(t)-ENE^{*}(t))\cdot s^{f}(t)dt
  \approx -\sum_{m   =1}^{M} \left (EPE^{*}(t_{m})-ENE^{*}(t_{m})\right)
\times\left \{ \textrm{exp}\left [ -s^{f}(t_{m-1})\cdot t_{m-1} \right ]- \textrm{exp}\left [ -s^{f}(t_{m})\cdot t_{m} \right ]\right \},
\end{aligned}
\end{equation}
where $EPE^{*}$ and $ENE^{*}$ are discounted expected positive exposure and discounted expected negative exposure, respectively, and $s^{f}(t)$ is the market funding spread. For a future time $t$, the $EPE$ and $ENE$ are given by:
\begin{equation}
    \begin{aligned}
      &EPE(t):=\mathbb{E}[E^{+}(t)|\mathfrak{F}_{0}],\\
      &ENE(t):=\mathbb{E}[E^{-}(t)|\mathfrak{F}_{0}],
    \end{aligned}
\end{equation}
where $x^{-}=\textrm{min}[x,0]$. In this work, we price a Bermudan option whose option value can never be negative. Therefore, the $ENE:=0$ and the FVA can be simplified to:
\begin{equation}
    FVA=-\int_{0}^{T} EE^{*}(t)\cdot s^{f}(t)dt.
\end{equation}
\\
\textbf{KVA} In general, the capital value adjustment represents a cost for a financial institute to meet the regulatory needs, and it measures the tail risk it faces \cite{ruiz2015complete}. Regulators set capital use restrictions or require banks to reach a certain capital threshold, at least implicitly charging capital for transactions. The formula of KVA is given by Gregory \cite{gregory2015xva}:
\begin{equation}
    KVA=-\int_{0}^{T}EK_{t}\cdot r_{c}\cdot DF_{t}dt,
\end{equation}
where $EK_{t}$ is the expected capital, $DF_{t}$ is the survival rate and  $r_{c}$ is the cost of holding the capital. As we mentioned in last section, for option pricing, the influence of the KVA can be neglected, thus we will not consider this term when we calculate XVA later.
\subsection*{Exposure of Bermudan option under L\'evy process}
From last section, we can see that both CVA and FVA calculation need to find exposure. And from (\ref{Et}), we know that exposure is decided by the option value. A Bermudan option is a special American-style  option, it can be early exercised at a restricted set of possible exercise dates. Let $\mathfrak{T}$ denote the set of exercises time:
\begin{equation}
    \mathfrak{T}=\left \{t_{1},t_{2},...,t_{M} \right \},
\end{equation}
where $M$ is the number of exercise times, and the time interval $\Delta t$ between each exercise time is equal. Because the dynamics of stock price is under burst, intermittent, disrupting fluctuations, the stock price can be considered as a process following the L\'evy process \cite{cartea2007fractional}.  Let $S_{t}$ be the price of underlying assets, which satisfies the following stochastic differential equation
\begin{equation}
    d(\textit{ln}S_{t})=(r-\nu)dt+dL_{t},
\end{equation}
with solution
\begin{equation}
   S_{t}=S_0e^{(r-\nu)t+\int_0^t dL_u},
\end{equation}
where $r$ is risk-free rate, $\nu$ is a convexity adjustment, $L_{t}$ is a L\'evy process, and $S_0$ is the initial price. In this work, let us consider a CGMY process $X_{t}$ defined by Carr \cite{carr2002fine}, it is a pure jump L\'evy process with L\'evy measure $W(dx)=\boldsymbol{w}_{CGMY}(x)$,
\begin{equation}
    \boldsymbol{w}_{CGMY}(x)=C\frac{e^{-G|x|}}{|x|^{1+Y}}\mathds{1}_{x<0}+C\frac{e^{-Mx}}{x^{1+Y}}\mathds{1}_{x>0},
\end{equation}
whose characteristic exponent can be obtained through L\'evy-Khintchine representation \cite{duanbook,cartea2007fractional}
\begin{equation}
    \Psi_{t}(z)=C\Gamma(Y)\left [ (M-iz)^{Y}-M^{Y}+(G+iz)^{Y}-G^{Y} \right ],
\end{equation}
where $z\in\mathbb{R}$, $\Gamma(x)$ is a Gamma function, $C>0$, $G\geq 0 $, $M\geq0$ and $Y<2$. The parameter $C$ measures the intensity of jumps, $G$ and $M$ control the skewness of distribution. For $Y\in[0,1]$, it means infinite activity process of finite variation, whereas for $Y\in(1,2)$, the process is infinity activity and infinity variation \cite{oosterlee2019mathematical}.\\For a CGMY process, the convexity adjustment is given by\cite{cartea2007fractional}:
\begin{equation}
    \nu=C\Gamma(Y)\left [ (M-1)^{Y}-M^{Y}+(G+1)^{Y}-G^{Y} \right ].
\end{equation}
At each exercise time, the payoff function $\varphi$ and continuous value $V^{c}$ are compared, and they are defined as:
\begin{equation}\label{payoff}
\varphi(S_{t_{m}})=\begin{cases}
(S_{t_{m}}-K)^{+}, & \textrm{for a call option,}\\
(K-S_{t_{m}})^{+}, & \textrm{for a put option,}
\end{cases}
\end{equation}
\begin{equation}\label{Cvalue}
    V^{c}(S_{t_{m}},t_{m})=e^{-r\Delta t}\mathbb{E}\left[ V(S_{t_{m+1}},t_{m+1})\mid S_{t_{m}}\right],
\end{equation}
where $K$ is strike price, $V(S_{t_{m+1}},t_{m+1})$ is the option value at time $t_{m+1}$. A natural assumption is that holder of the Bermudan option will exercise the option when the payoff value is higher than the continuous value at each $t_{m}$. Therefore, the value of Bermudan option satisfies the following term \cite{del2016efficient}:
\begin{equation}
     V(S_{t_{m}},t_{m})=\begin{cases}
\varphi(S_{t_{M}}) & \text{for}~ m=M,\\
\textrm{max}[V^{c}(S_{t_{m}},t_{m}),\varphi(S_{t_{m}})] & \text{for}~ m=1,2,\cdots,M-1,\\
V^{c}(S_{t_{0}},t_{0}) &\text{for}~ m=0.
\end{cases}
\end{equation}
According to \eqref{Et}, it is not difficult to find that the exposure equals to zero if the option is exercised, and the continuous value will be the exposure if no exercise. The exposure of Bermudan option at time $t_{m}$ is formulated as:
\begin{equation}
    E(t_{m})=V^{c}(S_{t_{m}},t_{m})\mathds{1}_{no-exercise},
\end{equation}
where $m\in[1,2,..,M-1]$. In addition, we have $E(t_{0})=V^{c}(S_{t_{0}},t_{0})$ and $E(t_{M})=0$.

\section*{Numerical Methods}

 In this section, we introduce two methods to compute the EE of the Bermudan option under the CGMY process. Both of them are combined with Monte Carlo simulation to obtain the option value and then the value adjustments can be found by the definition in Section II. A. For the Bermudan option, the option value and exposure depend on the underlying price $S_{t_{m}}$ at the exercise time point $t_{m}$. If the price paths are generated, the distribution of the future value can be computed. The Monte Carlo simulation of the CGMY process is quite complex. Based on Madan\cite{madan2006cgmy} and Sioutis \cite{sioutis2017calibration}, we can consider the CGMY process as a time-changed Brownian motion, i.e., the CGMY process can be written as
\begin{equation}
    X(t)=\frac{G-M}{2} \Upsilon(t)+B(\Upsilon(t)),
\end{equation}
where $\Upsilon(t)$ is a subordinator independent of the Brownian motion $B(t)$. By applying Rosinski \cite{rosinski2001series}  truncation method, the CGMY random variable can be simulated. More details can be found in Madan's work \cite{madan2006cgmy}.\\~\\
The general steps of whole algorithm are presented as follows:
\begin{steps}
    \newstep Simulate paths of underlying price $S_{t}$ under the CGMY model by Monte Carlo method. 
    \newstep \label{s2} Calculate continuous values and exercise values at each exercise time $t_{m}$ and terminal date, decide weather to exercise it or not.
    \newstep Find the exposure of each path if the option is not exercised, otherwise the exposure equals 0.
    \newstep Compute the CVA and FVA as defined in Section II. A.
\end{steps}
The rest parts of this section will explain the two numerical methods we will use in \emph{Step 2}.

\subsection*{Monte Carlo and finite difference of FPDE }

Recently, the fractional models have aroused numerous research interests in various fields, ranging from finance \cite{chen2021implicit,she2021novel}, neuroscience \cite{cai2019effects,cai2020state}, physics \cite{herrmann2011fractional,kirkpatrick2016fractional}, and so on \cite{yibook,laskin2000fractional,zaslavsky2007dynamics,li2020long,zhang2021pointwise,zhang2020runge}.  We here focus on the fractional model in option pricing. The finite difference method is a widely used method in option pricing. Combining with the results of Monte Carlo simulation, we can calculate the option values at different time points. The Monte Carlo simulation was introduced at the beginning of this section. In this part, we will establish the fully discrete numerical scheme for solving the FPDE related to the CGMY process. By using discrete energy estimate method, we will prove the convergence of the numerical scheme.\\
 Consider a European-style option under the CGMY process as defined in Section II. B,  Cartea \cite{cartea2007fractional} proves that option value $V$ satisfies the following FPDE:
\begin{equation}
     \begin{split}
     &\left[r+C\Gamma(-Y)(G^{Y}+M^{Y})\right]V(x,t)\\
     &=\frac{\partial V(x,t)}{\partial t}+(r-\nu)\frac{\partial V(x,t)}{\partial x}+C\Gamma (-Y)e^{-Gx} {_{-\infty}D_{x}^{Y}}\left( e^{Gx}V(x,t)\right)+C\Gamma (-Y)e^{Mx} {_{x}D_{\infty}^{Y}}\left( e^{-Mx}V(x,t)\right),
     \end{split}
\end{equation}
where $_{-\infty}D_{x}^{Y}$ and $_{x}D_{\infty}^{Y}$ are left and right Riemann-Liouville (RL) fractional derivatives, they are given by
\begin{equation}
    _{-\infty}D_{x}^{Y}f(x)=\frac{1}{\Gamma(p-Y)}\frac{\partial^{p}}{\partial x^{p}}\int_{-\infty}^{x}(x-y)^{p-Y-1}f(y)dy,
\end{equation}
\begin{equation}
    _{x}D_{\infty}^{Y}f(x)=\frac{(-1)^p}{\Gamma(p-Y)}\frac{\partial^{p}}{\partial x^{p}}\int_{-\infty}^{x}(y-x)^{p-Y-1}f(y)dy
\end{equation}
for $p-1\leq Y<p$, and $p$ is the smallest integer than $Y$. The left and right RL tempered fractional derivatives are defined as \cite{Sabzikar2015}
 $$ _{-\infty}D_{x}^{Y,G}f:=e^{-Gx}{ _{-\infty}D_{x}^{Y}}\left( e^{Gx}f\right),~~{_{x}D_{\infty}^{Y,M}}f:= e^{Mx} {_{x}D_{\infty}^{Y}}\left( e^{-Mx}f\right).$$
Furthermore, for an American-style option, it becomes a free-boundary problem \cite{guo2016valuation}. Considering the optimal-exercise boundary and payoff function (see equation (28) and (29) in Guo's work\cite{guo2016valuation}), the American option value under the CGMY model satisfies:
\begin{equation}\label{americanop}
     \begin{split}
     &\left[r+C\Gamma(-Y)(G^{Y}+M^{Y})\right]V(x,t)\\
     &\geq\frac{\partial V(x,t)}{\partial t}+(r-\nu)\frac{\partial V(x,t)}{\partial x}+C\Gamma (-Y)e^{-Gx} {_{-\infty}D_{x}^{Y}}\left( e^{Gx}V(x,t)\right)+C\Gamma (-Y)e^{Mx} {_{x}D_{\infty}^{Y}}\left( e^{-Mx}V(x,t)\right),
     \end{split}
\end{equation}
for any time $t \in [t_{0},T]$ the option can be exercised. A Bermudan option can be seen as a discrete case of American option, i.e., for each exercise time, we can solve \eqref{americanop} as an equality, and the Bermudan option value will take the maximum value between this value and the exercise value.\\
\subsubsection*{Finite difference scheme for FPDE related with CGMY process}
In the following approximation method, the unbounded spatial domain is truncated into a bounded one, $x\in(x_L, x_R).$ Now, we present the finite difference scheme for solving the FPDE related with the CGMY process for a European call option,
\begin{equation}\label{model-1}
\begin{split}
\begin{cases}
&rV(x,t)=\frac{\partial V(x,t)}{\partial t}+(r-\nu)\frac{\partial V(x,t)}{\partial x}+C\Gamma (-Y)[ _{x_L}D_{x}^{Y,G}V(x,t)  -G^{Y}V(x,t)]\\
&~~~~~~~~~~~~~~~~~~~+C\Gamma (-Y)[ {_{x}D_{x_R}^{Y,M}}V(x,t)-M^{Y}V(x,t)], (x,t)\in(x_L,x_R)\times(0,T),\\
&V(x_L,t)=0,~V(x_R,t)=e^{x_R}-Ke^{-r(T-t)},~t\in(0,T),\\
&V(x,T)=(e^x-K)^{+},~~ x\in(x_L, x_R),
\end{cases}
\end{split}
\end{equation}
where $ _{x_L}D_{x}^{Y,G}V(x,t)=e^{-Gx}{ _{x_L}D_{x}^{Y}}\left( e^{Gx}V(x,t)\right),~{_{x}D_{x_R}^{Y,M}}V(x,t)= e^{Mx} {_{x}D_{x_R}^{Y}}\left( e^{-Mx}V(x,t)\right).$
\\
We divide the temporal domain into $N$ parts by the grid points $t_j=T-j\tau ~(0 \leq j \leq N_t),$ where the temporal stepsize $\tau=\frac{T}{N_t}.$ The spatial domain is divided into $N_x$ parts by the mesh points $x_n=x_L+nh~(0 \leq n \leq N_x),$ where the spatial stepsize $h=\frac{x_R-x_L}{N_x}.$  The temporal domain is covered by $\Omega_{\tau}= \{t_j| 0\leq j \leq N_t \}$, and the spatial domain is covered by $\Omega_h= \{x_n| 0\leq n \leq N_x \}$. Let $\mathcal {V}_h=\{v|v=\{v_{n}^j|~0 \leq j \leq N_t,~ 0 \leq n \leq N_x\} \}$ be grid function space defined on
$\Omega_\tau\times \Omega_h$. For the grid function $v \in \mathcal {V}_h,$ we have the following notations:
$$ v_{n}^{j+\frac{1}{2}}=\frac{v_{n}^{j+1}+v_{n}^j}{2},~~   \delta_t v_{n}^{j+\frac{1}{2}}=\frac{v_{n}^{j+1}-v_{n}^j}{-\tau},~~
\delta_{x0} v_{n}^{j+\frac{1}{2}}= \frac{v_{n+1}^{j+\frac{1}{2}}-v_{n-1}^{j+\frac{1}{2}}}{2 h}.$$
First, we approximate the first-order space derivative by using the second-order central difference operator, and for the tempered fractional order
derivatives, we turn to the following useful lemma. For conciseness, we let the parameter $\lambda$ refer to $G$ and $M,$ respectively.
 \begin{lemma}\cite{lican2015}\label{two-ord}
  The Y-th order left and right RL tempered fractional derivatives of $V(x)$ at point $x_n$ can be approximated by the tempered-WSGD operators:
 \begin{eqnarray}\label{WSGD-2ord}
    && _{x_L}D_{x}^{Y,\lambda}V(x_n)  -\lambda^{Y}V(x_n):= _{L}\mathfrak{D}_{h}^{Y,\lambda}V(x_n)+\mathcal {O}(h^{2})=\frac{1}{h^{Y}}\left[\sum\limits_{l=0}^{n+1}g_{l,\lambda}^{(Y)}V(x_{n-l+1})-\phi(\lambda)V\left(x_{n}\right)\right]+\mathcal {O}(h^{2}),\\
    && _{x}D_{x_R}^{Y,\lambda}V(x_n)-\lambda^{Y}V(x_n):=_{R}\mathfrak{D}_{h}^{Y,\lambda}V(x_n)+\mathcal {O}(h^{2})=\frac{1}{h^{Y}}\left[\sum\limits_{l=0}^{N_x-n+1}g_{l,\lambda}^{(Y)}V(x_{n+l-1})-\phi(\lambda)V\left(x_{n}\right)\right]+\mathcal {O}(h^{2}),
  \end{eqnarray}
  here
  \begin{equation}
       \phi(\lambda)=\left(\gamma_{1} e^{h \lambda}+\gamma_{2}+\gamma_{3} e^{-h \lambda}\right)\left(1-e^{-h \lambda}\right)^{Y}
  \end{equation}and the weights are given by
\begin{equation}\label{weight-g}
\begin{split}
\begin{cases}
&g_{0, \lambda}^{(Y)}=\gamma_{1} \omega_{0} e^{h \lambda}, g_{1, \lambda}^{(Y)}=\gamma_{1} \omega_{1}+\gamma_{2} \omega_{0}, \\
&g_{l, \lambda}^{(Y)}=\left(\gamma_{1} \omega_{l}+\gamma_{2} \omega_{l-1}+\gamma_{3} \omega_{l-2}\right) e^{-(l-1) h \lambda}, ~l \geq 2,\\
&\omega_{0}=1, \omega_{l}=\left(1-\frac{1+Y}{l}\right) \omega_{l-1},~ l \geq 1,
\end{cases}
\end{split}
\end{equation}
and the parameters $\gamma_1,~\gamma_2$ and $\gamma_3$ admit the following linear system
 \begin{equation}\label{gamma3}
\begin{split}
\begin{cases}
&\gamma_1=\frac{Y}{2}+\gamma_3, \\
&\gamma_2=\frac{2-Y}{2}-2\gamma_3,
\end{cases}
\end{split}
\end{equation}
here $\gamma_3$ is the free variable.
\end{lemma}

Second, we employ the  Crank-Nicolson scheme to discrete the time derivative. Then, we arrive at
\begin{eqnarray}\label{exact}
&&\delta_t V_{n}^{j+\frac{1}{2}}+(r-\nu)\delta_{x0} V_{n}^{j+\frac{1}{2}}+C\Gamma (-Y)[_{L} \mathfrak{D}_{h}^{Y, G} V_{n}^{j+\frac{1}{2}}+  _{R} \mathfrak{D}_{h}^{Y, M} V_{n}^{j+\frac{1}{2}}] \nonumber\\
&&-r V_{n}^{j+\frac{1}{2}}=R_n^{j+\frac{1}{2}},~~~~0 \leq j \leq N_t-1,~~~1\leq n \leq N_x-1,
\end{eqnarray}
and there exists a constant $C_R$ such that the truncation error
\begin{equation}\label{trunc1}
|R_n^{j+\frac{1}{2}}|\leq C_R(\tau^2+h^2).
\end{equation}
Let $v_{n}^{j+\frac{1}{2}}$ denote the approximate solution to ${V}_{n}^{j+\frac{1}{2}}$ and drop the truncation error. Finally, we obtain the fully discrete finite difference scheme:
\begin{eqnarray}\label{numeric}
\delta_t v_{n}^{j+\frac{1}{2}}+(r-\nu)\delta_{x0} v_{n}^{j+\frac{1}{2}}+C\Gamma (-Y)[_{L} \mathfrak{D}_{h}^{Y, G} v_{n}^{j+\frac{1}{2}}+  _{R} \mathfrak{D}_{h}^{Y, M} v_{n}^{j+\frac{1}{2}}]-r v_{n}^{j+\frac{1}{2}} =0,~~~~0 \leq j \leq N_t-1,~~~1\leq n \leq N_x-1.
\end{eqnarray}
The initial-boundary conditions for a call option are discretized as
$$v_{n}^{0}=(e^{x_n}-K)^{+},~~ n=1,2, \ldots, N_x-1, $$
$$v_{0}^{j}=0,~~v_{N_x}^{j}=e^{x_R}-Ke^{-r(T-t_j)},~~ j=0,1, \ldots, N_t.$$
We next move on to the analysis of convergence for the scheme \eqref{numeric}.
For any grid functions $v, u \in \mathcal{V}_h$, we introduce the discrete inner product
$$\langle v,u\rangle=h\sum\limits_{n=1}^{N_x-1}v_{n} u_{n},$$
and corresponding induced norm
$$||v||=\sqrt{\langle v,v\rangle}.$$
\subsubsection*{Convergence analysis of finite difference scheme}
We propose two lemmas below, which are essential to the analysis of convergence.
\begin{lemma}\label{keylemma-phi}
For $Y\in(1,2)$ and $\lambda\geq0,$ if $\gamma_3\geq-\frac{Y}{2}$, we have
\begin{eqnarray}
\phi(\lambda)\geq0,
\end{eqnarray}
where $\phi(\lambda)$ is defined in Lemma 1.
\end{lemma}
\begin{proof}
According to the definition of $\phi(\lambda),$ we have
$$\phi(\lambda)=\left(\gamma_{1} e^{h \lambda}+\gamma_{2}+\gamma_{3} e^{-h \lambda}\right)\left(1-e^{-h \lambda}\right)^{Y}.$$
For $h\lambda \geq0,$ the function $\left(1-e^{-h \lambda}\right)^{Y}\geq0.$ Then our goal is to prove
$$\phi_1(\lambda)=\gamma_{1} e^{h \lambda}+\gamma_{2}+\gamma_{3} e^{-h \lambda} \geq 0.$$
 Substituting the expressions of parameters $\gamma_1,~\gamma_2$ (see \eqref{gamma3}) into $\phi_1(\lambda),$ we get
\begin{eqnarray}
  \phi_1(\lambda)=\Big(\frac{Y}{2}+\gamma_3\Big) e^{h \lambda}+\frac{2-Y}{2}-2\gamma_3+\gamma_{3} e^{-h \lambda}=\gamma_{3}\big(e^{h \lambda}+e^{-h \lambda}-2\big)+\frac{Y}{2}\big(e^{h \lambda}-1\big)+1.
\end{eqnarray}
 Noting that $e^{h \lambda}+e^{-h \lambda}-2 \geq0.$
 When $e^{h \lambda}+e^{-h \lambda}-2 =0,$ it holds that $\phi_1(\lambda)>0.$
 When $e^{h \lambda}+e^{-h \lambda}-2 >0,$ if
 \begin{equation}\label{gamma3-1}
 \gamma_3\geq \frac{\frac{Y}{2}\big(1-e^{h \lambda}\big)-1}{e^{h \lambda}+e^{-h \lambda}-2 }:=f(h\lambda),
 \end{equation}
then we have $\phi_1(\lambda)\geq0.$ To make sure the formula \eqref{gamma3-1} always holds, we have to compute the maximum of function $f(h\lambda).$ One can check that $f(h\lambda)$ is a monotone increasing function, thus
 \begin{eqnarray}\label{gamma3-2}
\max_{h\lambda\geq0} f(h\lambda)=\lim_{h\lambda\rightarrow +\infty}f(h\lambda):=\lim_{h\lambda\rightarrow +\infty}\frac{\frac{Y}{2}\big(1-e^{h \lambda}\big)-1}{e^{h \lambda}+e^{-h \lambda}-2 }=-\frac{Y}{2}.
 \end{eqnarray}
 That is, if $\gamma_3\geq-\frac{Y}{2}$, we have $\phi_1(\lambda)\geq0.$\\In short, for $Y\in(1,2)$ and $\lambda\geq0,$ if $\gamma_3\geq-\frac{Y}{2}$, we have
$\phi(\lambda)\geq0.$ $\Box$
\end{proof}

\begin{lemma}\label{keylemma}
For any function $v \in \mathcal{V}_h$, and $Y\in(1,2).$ If $\gamma_3\geq-\frac{Y}{2}$ and $0\leq\lambda\leq \frac{c_0}{h},$  then we have
\begin{eqnarray}
&&\langle {_{L} \mathfrak{D}_{h}^{Y, \lambda}} v,v\rangle \leq c_2\Big(\frac{2}{x_R-x_L}\Big)^Y||v||^2,\label{keylemma-1}\\
&& \langle {_{R} \mathfrak{D}_{h}^{Y, \lambda}} v,v\rangle\leq c_2\Big(\frac{2}{x_R-x_L}\Big)^Y||v||^2,\label{keylemma-2}
\end{eqnarray}
where $c_0$ and $c_2$ are positive constants independent of $h$ and $\tau.$
\end{lemma}
\begin{proof}
Recalling the definition of ${_{L} \mathfrak{D}_{h}^{Y, \lambda}} $ in \eqref{WSGD-2ord} and applying Lemma \ref{keylemma-phi}, we have
\begin{eqnarray}\label{key-1}
&&\langle {_{L} \mathfrak{D}_{h}^{Y, \lambda}} v,v\rangle\nonumber\\
&=&h\sum\limits_{n=1}^{N_x-1}\frac{1}{h^{Y}}\left[\sum\limits_{l=0}^{n+1}g_{l,\lambda}^{(Y)}v_{n-l+1}-\phi(\lambda)v_n\right]v_n,\nonumber\\
& \leq&\frac{h}{h^Y} \left[  g_{1,\lambda}^{(Y)}\sum\limits_{n=1}^{N_x-1} v_n^2+\big(g_{0,\lambda}^{(Y)}+g_{2,\lambda}^{(Y)}\big)\sum\limits_{n=1}^{N_x-2} v_n v_{n+1}+\sum\limits_{l=3}^{N_x-1}g_{l,\lambda}^{(Y)}\sum\limits_{n=1}^{N_x-l}v_{n+l-1}v_n \right]\nonumber\\
& \leq& \frac{1}{h^Y} \left[  g_{1,\lambda}^{(Y)}||v||^2+\big(g_{0,\lambda}^{(Y)}+g_{2,\lambda}^{(Y)}\big)h\sum\limits_{n=1}^{N_x-2}\frac{v_n^2+v_{n+1}^2}{2} \right.\left.+  \sum\limits_{l=3}^{N_x-1}g_{l,\lambda}^{(Y)}h\sum\limits_{n=1}^{N_x-l}\frac{v_{n+l-1}^2+v_n^2 }{2}\right]\nonumber\\
& \leq& \frac{1}{h^Y} \sum\limits_{l=0}^{N_x-1}|g_{l,\lambda}^{(Y)}| ||v||^2.
\end{eqnarray}
According to the definition of $g_{l,\lambda}^{(Y)}$ in \eqref{weight-g} and noting $h\lambda\leq c_0, ~Y\in(1,2)$, there exists an appropriate positive constant $c_1$  such that
\begin{eqnarray}\label{key-2}
&& \sum\limits_{l=0}^{N_x-1}|g_{l,\lambda}^{(Y)}| = |\gamma_{1} \omega_{0} e^{h \lambda}|+| \gamma_{1} \omega_{1}+\gamma_{2} \omega_{0} |
\nonumber +\sum\limits_{l=2}^{N_x-1}\left(\gamma_{1} \omega_{l}+\gamma_{2} \omega_{l-1}+\gamma_{3} \omega_{l-2}\right) e^{-(l-1) h \lambda}\nonumber\\
& \leq& c_1\left[|\omega_{0}|+|\omega_{1}|+\sum\limits_{l=2}^{N_x-1} \big(|\omega_{l}|+| \omega_{l-1}|+| \omega_{l-2}|\big)\right]\nonumber\\
& \leq &3c_1\sum\limits_{l=0}^{\infty}|\omega_{l}|,
\end{eqnarray}
where we have used the fact that the parameters $\gamma_1$ and $\gamma_2$ are constants for fixed $Y$ and $\gamma_3.$

By virtue of the properties of the weights $\omega_{l}$ \cite{Dimitrov2014},
\begin{equation*}
 |\omega_{l}|\leq \frac{Y2^{Y+1}}{(l+1)^{Y+1}},~~~l\geq 0,
\end{equation*}
it holds that
\begin{eqnarray}\label{key-3}
 \sum\limits_{l=0}^{\infty}|\omega_{l}|
\leq \sum\limits_{l=0}^{{N_x} -1}\frac{Y2^{Y+1}}{(l+1)^{Y+1}}+\sum\limits_{l={N_x} }^{\infty}\frac{Y2^{Y+1}}{(l+1)^{Y+1}}.
\end{eqnarray}

Noting that for $x\geq0,$  $\frac{1}{(x+1)^{Y+1}}$ is a monotone decreasing function, we arrive at
\begin{equation}\label{key-4}
 \sum\limits_{l={N_x} }^{\infty} \frac{1}{(l+1)^{Y+1}}\leq  \sum\limits_{l={N_x} }^{\infty} \int_{l}^{l+1}\frac{1}{(x+1)^{Y+1}} dx=\frac{1}{Y(N_x +1)^Y}.
\end{equation}

Therefore, together with \eqref{key-1}-\eqref{key-4}, there exists a positive constant $c_2$ such that
\begin{eqnarray}\label{key-5}
&&\langle {_{L} \mathfrak{D}_{h}^{Y, \lambda}} v,v\rangle\nonumber\\
&&
\leq \frac{3c_1}{h^Y}  \sum\limits_{l=0}^{\infty}|\omega_{l}| ||v||^2\nonumber\\
&& \leq \frac{3c_1}{h^Y} \Big[ Y2^{Y+1}\Big(N_x + \frac{1}{Y(N_x +1)^Y}\Big) \Big]||v||^2\nonumber\\
&&\leq\frac{c_2}{h^Y}\Big(\frac{2}{N_x +1}\Big)^Y||v||^2=c_22^Yh^{-Y}N_x^{-Y}\frac{N_x^Y}{(N_x+1)^Y}||v||^2\nonumber\\
&&\leq c_2 \Big(\frac{2}{x_R-x_L} \Big)^Y||v||^2.
\end{eqnarray}

Similarly, we have
\begin{eqnarray}
 \langle {_{R} \mathfrak{D}_{h}^{Y, \lambda}} v,v\rangle\leq c_2\Big(\frac{2}{x_R-x_L}\Big)^Y||v||^2.
\end{eqnarray}
$\Box$
\end{proof}

\begin{lemma}\cite{sunbook}\label{grw}
Suppose that $\{ F^j|~ j\geq 0\}$ is a non-negative sequence and satisfies
$$F^{j+1} \leq(1+c \tau) F^{j}+\tau g, ~~ j=0,1,2, \cdots, $$
where $c$ and $g$ are two non negative constants.
Then we have
$$F^{j} \leq e^{c j \tau}\left(F^{0}+\frac{g}{c}\right), ~~ j=1,2,3, \cdots$$
\end{lemma}

For $j=1,2, \ldots, N_t,$ let $\epsilon^{j}=V^{j}-v^{j},$ where $V^{j}=\left(V_{1}^{j}, V_{2}^{j}, \ldots, V_{N_x-1}^{j}\right)^{T}$ and $v^{j}=\left(v_{1}^{j}, v_{2}^{j}, \ldots, v_{N_x-1}^{j}\right)^{T}.$ We have the following convergence result.
\begin{theorem}\label{convergence}
For $Y\in(1,2)$, $\gamma_3\geq-\frac{Y}{2}$, $0\leq G\leq\frac{c_0}{h} $ and $0\leq M\leq\frac{c_0}{h}$, we have
$$\left\|\epsilon^{j}\right\| \leq \hat{c}\left(h^{2}+\tau^{2}\right),~~ j=1,2, \ldots, N_t,$$
where $\hat{c}$ denotes a positive constant independent of $h$ and $\tau.$
\end{theorem}
\begin{proof}
  Subtracting \eqref{numeric} from \eqref{exact}, we get the following error equation

 \begin{eqnarray}\label{erreqn}
\delta_t \epsilon_{n}^{j+\frac{1}{2}}+(r-\nu)\delta_{x0} \epsilon_{n}^{j+\frac{1}{2}}+C\Gamma (-Y)[_{L} \mathfrak{D}_{h}^{Y, G} \epsilon_{n}^{j+\frac{1}{2}}+  _{R} \mathfrak{D}_{h}^{Y, M} \epsilon_{n}^{j+\frac{1}{2}}]-r \epsilon_{n}^{j+\frac{1}{2}} =R_n^{j+\frac{1}{2}},~~~~0 \leq j \leq N_t-1,~~~1\leq n \leq N_x-1.
\end{eqnarray}

  Taking inner product $\langle \cdot,\cdot\rangle$ on both sides of Eq. \eqref{erreqn} with $\epsilon^{j+\frac{1}{2}}$, we have
\begin{eqnarray}\label{inner-eqn}
&&~~~\langle\delta_t \epsilon^{j+\frac{1}{2}},\epsilon^{j+\frac{1}{2}}\rangle+(r-\nu)\langle\delta_{x0} \epsilon^{j+\frac{1}{2}},\epsilon^{j+\frac{1}{2}}\rangle C\Gamma (-Y)[\langle_{L} \mathfrak{D}_{h}^{Y, G} \epsilon^{j+\frac{1}{2}},\epsilon^{j+\frac{1}{2}}\rangle+ \langle _{R} \mathfrak{D}_{h}^{Y, M} \epsilon^{j+\frac{1}{2}},\epsilon^{j+\frac{1}{2}}\rangle]\nonumber\\
&&=r \langle \epsilon^{j+\frac{1}{2}},\epsilon^{j+\frac{1}{2}}\rangle+\langle R^{j+\frac{1}{2}},\epsilon^{j+\frac{1}{2}}\rangle.
\end{eqnarray}
For the first term of the left-hand side (LHS) of \eqref{inner-eqn}, it holds that

 \begin{eqnarray}\label{lhs-1}
\langle\delta_t \epsilon^{j+\frac{1}{2}},\epsilon^{j+\frac{1}{2}}\rangle=\langle \frac{\epsilon^{j+1}-\epsilon^j}{-\tau},\frac{\epsilon^{j+1}+\epsilon^j}{2}\rangle=\frac{||\epsilon^{j+1}||^2-||\epsilon^{j}||^2}{-2\tau}.
\end{eqnarray}

For the second term of the LHS of \eqref{inner-eqn}, noting that $\epsilon_0^j=\epsilon_{N_x}^j=0~(0\leq j \leq N_t),$ we have
 \begin{eqnarray}\label{lhs-2}
(r-\nu)\langle\delta_{x0} \epsilon^{j+\frac{1}{2}},\epsilon^{j+\frac{1}{2}}\rangle=(r-\nu)h\sum\limits_{n=1}^{{N_x} -1} \Big(\frac{\epsilon_{n+1}^{j+\frac{1}{2}}-\epsilon_{n-1}^{j+\frac{1}{2}}}{2 h}\Big)\epsilon_n^{j+\frac{1}{2}}=0.
\end{eqnarray}

Applying Lemma \ref{keylemma} to the third term of the LHS of \eqref{inner-eqn} yields
 \begin{eqnarray}\label{lhs-3}
&&C\Gamma (-Y)[\langle_{L} \mathfrak{D}_{h}^{Y, G} \epsilon^{j+\frac{1}{2}},\epsilon^{j+\frac{1}{2}}\rangle+ \langle _{R} \mathfrak{D}_{h}^{Y, M} \epsilon^{j+\frac{1}{2}},\epsilon^{j+\frac{1}{2}}\rangle]\nonumber\\
&&\leq 2 C\Gamma (-Y)c_2\Big(\frac{2}{x_R-x_L}\Big)^Y||\epsilon^{j+\frac{1}{2}}||^2.
\end{eqnarray}
Combining with \eqref{lhs-1}-\eqref{lhs-3} and using the $\varrho-$inequality ($bd\leq\frac{1}{2\varrho}b^2+\frac{\varrho}{2} d^2$), we arrive at
 \begin{eqnarray}\label{tran-1}
&&||\epsilon^{j+1}||^2-||\epsilon^{j}||^2-2\tau C\Gamma (-Y)c_2\Big(\frac{2}{x_R-x_L}\Big)^Y\Big(||\epsilon^{j+1}||^2+||\epsilon^{j}||^2\Big)\nonumber\\
&&\leq-2\tau r ||\epsilon^{j+\frac{1}{2}}||^2-2\tau\langle R^{j+\frac{1}{2}},\epsilon^{j+\frac{1}{2}}\rangle \nonumber\\
&&\leq -2\tau r ||\epsilon^{j+\frac{1}{2}}||^2+ \tau \Big( \varrho ||R^{j+\frac{1}{2}}||^2+\frac{1}{\varrho}||\epsilon^{j+\frac{1}{2}}||^2\Big).
\end{eqnarray}
Taking $\varrho=\frac{1}{2r}$, we can eliminate the term $||\epsilon^{j+\frac{1}{2}}||^2$. After some calculations, then we obtain

  \begin{eqnarray}\label{tran-2}
&&\Big[ 1-2\tau C\Gamma (-Y)c_2\Big(\frac{2}{x_R-x_L}\Big)^Y\Big]||\epsilon^{j+1}||^2\nonumber\\
&&\leq\Big[ 1+2\tau C\Gamma (-Y)c_2\Big(\frac{2}{x_R-x_L}\Big)^Y\Big]||\epsilon^{j}||^2+\frac{\tau}{2r}||R^{j+\frac{1}{2}}||^2.
\end{eqnarray}

Letting $2\tau C\Gamma (-Y)c_2\Big(\frac{2}{x_R-x_L}\Big)^Y\leq \frac{1}{3}$, we have
  \begin{eqnarray}\label{tran-3}
&&||\epsilon^{j+1}||^2\leq\Big[ 1+6\tau C\Gamma (-Y)c_2\Big(\frac{2}{x_R-x_L}\Big)^Y\Big]||\epsilon^{j}||^2\nonumber\\
&&~~~~~~~~~~~~~~~~~+\frac{3\tau}{4r}||R^{j+\frac{1}{2}}||^2.
\end{eqnarray}

Applying Lemma \ref{grw} to \eqref{tran-3} results in
  \begin{eqnarray}\label{tran-4}
||\epsilon^{j}||^2\leq \text{exp}\left ( 6T C\Gamma (-Y)c_2\Big(\frac{2}{x_R-x_L}\Big)^Y\right ) \frac{(x_R-x_L)^{1+Y}C_R^2}{2^{3+Y}rC\Gamma (-Y)c_2}  (\tau^2+h^2)^2.
\end{eqnarray}
That is,
\begin{equation}
 \left\|\epsilon^{j}\right\| \leq \hat{c}\left(h^{2}+\tau^{2}\right),~~ j=1,2, \ldots, N_t,
\end{equation}
where $$\hat{c}=\sqrt{ \text{exp}\left (6T C\Gamma (-Y)c_2\Big(\frac{2}{x_R-x_L}\Big)^Y \right) \frac{(x_R-x_L)^{1+Y}C_R^2}{2^{3+Y}rC\Gamma (-Y)c_2} },$$ which is a positive constant independent of $h$ and $\tau.$
$\Box$
\end{proof}
\subsection*{Monte Carlo-COS method}
The COS method is an efficient option pricing method proposed by Fang \cite{fang2009novel}. The central idea of this method is to estimate the probability density function via a Fourier cosine expansion. Since the characteristic function of the L\'evy process has a closed-form relation with the Fourier cosine series coefficients, the COS method can be applied to many complex underlying price processes, including the CGMY process. Incorporated with Monte Carlo method, the so-called Monte Carlo-COS (MC-COS) method is a benchmark approach to compute the exposure of the Bermudan option \cite{shen2013benchmark}.  \\
Recall the Section II.B, for a Bermudan option, if we define $x=\text{log}(\frac{S_{t_{m}}}{K})$, $y=\text{log}(\frac{S_{t_{m+1}}}{K})$, and denote $\hat{V}(y)=V(K e^y)=V(S)$, then the continuous value \eqref{Cvalue} is transformed to:
\begin{eqnarray}
\label{VCdensity}
      V^{c}(S_{t_{m}},t_{m})=e^{-r\Delta t}\mathbb{E}\left[ \hat{V}(y,t_{m+1})\mid x\right] =e^{-r\Delta t}\int_{\mathbb{R}}\hat{V}(y,t_{m+1})f(y|x)dy,
\end{eqnarray}
where $f(y|x)$ is the conditional density function of $y$. Based on Fourier cosine expansion, Fang \cite{fang2009novel} proves that \eqref{VCdensity} can be approximated by:
\begin{eqnarray}
\label{Vbar}
    \bar{V}^{c}(S_{t_{m}},t_{m})=e^{-r\Delta t}{\sum_{k=0}^{N-1}}'\text{Re}\left \{ \phi_{X}(\frac{k\pi}{b-a})e^{-ik\pi\frac{a}{b-a}} \right \}\times H_{k}(t_{m+1}),
\end{eqnarray}
where ${\sum}'$ means that the first term of the summation is weighted by $1/2$, $\text{Re}(x)$  is to take real part of $x$, and $\phi_{X}$ is the characteristic function of L\'evy process $L_t$, for the CGMY process,
\begin{equation}
\phi_{CGMY}(z,t):=e^{ tC\Gamma(-Y)\left( (M-iz)^{Y}-M^{Y}+(G+iz)^{Y}-G^{Y}\right)}.
\end{equation}
And $H_{k}$ is the Fourier cosine series coefficients of $\hat{V}(y)$:
\begin{equation}
    H_{k}(t_{m+1}):= \frac{2}{b-a}\int_{a}^{b}\hat{V}(y,t_{m+1})\text{cos}\left( k\pi\frac{y-a}{b-a}\right)dy,
\end{equation}
here $[a,b]$ is the the truncation rage of integration. By the definition in Oosterlee \cite{oosterlee2019mathematical} and Fang's\cite{fang2009novel} works, the integration range can be defined as:
\begin{equation}
    [a,b]:=\left[ (x+\zeta_{1})-L\sqrt{\zeta_{2}+\sqrt{\zeta_{4}}},(x+\zeta_{1})+L\sqrt{\zeta_{2}+\sqrt{\zeta_{4}}}\right],
\end{equation}
where $L\in [6,12]$ is a user-decided parameter to control the tolerance level, and $\zeta_{1},\zeta_{2},\zeta_{4}$ are cumulants of  L\'evy process, for the CGMY process, they are defined as:
\begin{equation}
\begin{split}
    \zeta_{1}=rt+Ct\Gamma(1-Y)(M^{Y-1}-G^{Y-1}),\\
    \zeta_{2}=Ct\Gamma(2-Y)(M^{Y-2}+G^{Y-2}),\\
    \zeta_{4}=C\Gamma(4-Y)t(M^{Y-4}+G^{Y-4}).
    \end{split}
\end{equation}
\\
In order to compute \eqref{Vbar}, we need to find the Fourier cosine series coefficients $H_{k}$ first. As in Fang's paper\cite{fang2009novel}, for the Bermudan option, we need to consider the early exercise time points. Let $x^{*}(t_{m})$ be the point where the continuous value equals the payoff function, it can be found by Newton root finding algorithm. Then $H_{k}(t_{m})$ is split into two interval: $[a,x^{*}(t_{m})]$ and $[x^{*}(t_{m}),b]$. For a call option, in the case that $a<0<b$, we obtain \cite{shen2013benchmark}:
\begin{equation}\label{HkC}
H_{k}(t_{m+1})=\frac{2}{b-a}(\psi_{k}(a,x^{*}(t_{m+1}))+\chi_{k}(x^{*}(t_{m+1}),b))
\end{equation}
for $m=M-2,...,0$ and at $t_{M}=T$,
\begin{equation}
H_{k}(t_{M})=\frac{2}{b-a}\chi_{k}(0,b),
\end{equation}
where the cosine series coefficients $\chi_{k}$ and $\psi_{k}$ on an integration interval $[c,d]\subset[a,b]$ are given by:
\begin{equation}
    \chi_{k}(c,d):=\int_{c}^{d}e^{y}\text{cos}\left( k\pi\frac{y-a}{b-a}\right)dy,
\end{equation}
\begin{equation}
    \psi_{k}(c,d):=\int_{c}^{d}V^{c}(t_{m+1})\text{cos}\left( k\pi\frac{y-a}{b-a}\right)dy
\end{equation}
for $k=0,1,...,N-1$ and $m=0,1,..,M-1$,  $\chi_{k}$ has an analytical solution and $\psi_{k}$  can be  also approximated by a same method, the derivation can be found in Fang\cite{fang2009novel} and Kienitz's\cite{kienitz2013financial} papers. Similarly, for a put option, we have
\begin{equation}\label{HkP}
H_{k}(t_{m+1})=\frac{2}{b-a}(\chi_{k}(a,x^{*}(t_{m+1}))+\psi_{k}(x^{*}(t_{m+1}),b))
\end{equation}
for $m=M-2,...,0$ and at $t_{M}=T$,
\begin{equation}
H_{k}(t_{M})=\frac{2}{b-a}\chi_{k}(a,0).
\end{equation}
\\
Let us consider the Bermudan call option. For each simulated path, we calculate the truncation interval~$[a,b]$ first. Then the Fourier cosine coefficients $H_{k}$ can be obtained by  \eqref{HkC} for a call option. For the terminal time $t_{M}=T$, option value $V(S_{t_{M}},t_{M})=\varphi(S_{t_{M}})$. For other time steps, applying the backward induction, we can get the approximation of continuous value \eqref{Vbar} at $t_{m-1}$ from the value at $t_{m}$. Finding the minimum time point $\tau$ if it exists such that $\varphi(S_{\tau)}\geq V^{c}(S_{\tau}, \tau)$. Then the option value at each time step becomes $V(S_{t_{m}},t_{m})=\text{max}\left(\varphi(S_{t_{m}}), V^{c}(S_{t_{m}},t_{m}) \right)$ and $V(S_{t},t)=0$ for $t>\tau$. Setting the exposure $E(t_{m})=(V(S_{t_{m}},t_{m}),0)^+$, considering all simulated paths, we can get the EE and value adjustments defined in Section 2.\\
As a benchmark method, the MC-COS method has a high accuracy, and there are many discussions about the errors caused by truncated integration ranges, quadrature and propagation, such as Fang \cite{fang2009novel} and Oosterlee\cite{oosterlee2019mathematical}. The disadvantage of this approach is the computing time. Compared with MC-FF, the MC-COS is significantly slower when more Monte Carlo paths are used.

\section*{Numerical Results}
In this section, we will present  several numerical results for value adjustments of the Bermudan option under the CGMY process. The parameters of four different examples can be found in Table I. The initial price $S_{0}=40,$ and the risk-free rate $r=0.05$ are not changed in every experiment. The parameter $C$ of CGMY process may vary in different examples. And the default parameters of CGMY process are set as: $C=1,~G=25,~M=26$ and $Y=1.5.$ In all cases, experiments have been performed by using MATLAB on an Intel(R) Core(TM) i7-8700 CPU computer.  \\
\begin{table}\label{parameters}
	\caption{Parameters of each example.}
\centering
	\begin{tabular}{ccccc}
	\toprule
		~&Example 1&Example 2& Example 3 &Example 4\\
		\cmidrule(r){2-5}
		Strike Price $K$&50&40&50&40\\
		Expiry Time $T$&1&1&0.5 &0.5\\
		Exercise Times&50&50&30&30 \\
		$C$&1&1&0.5&0.5 \\
		\bottomrule
	\end{tabular}

\end{table}

 To reduce the impact of noise of Monte Carlo simulations, we generated $10^{4}$ paths for each example. Figure 1 presents a distribution plot of the terminal price $S_{T}$, which is  generated by Monte Carlo method. It can be seen that most values are within $200$. Every Monte Carlo price path has a similar distribution. Therefore, we set the price boundary of finite difference method and COS method to $400$ to ensure that more than $99\%$ of paths can be used, while also avoiding some noise. \\
\begin{figure}
\centering
\includegraphics[width=0.4\textwidth]{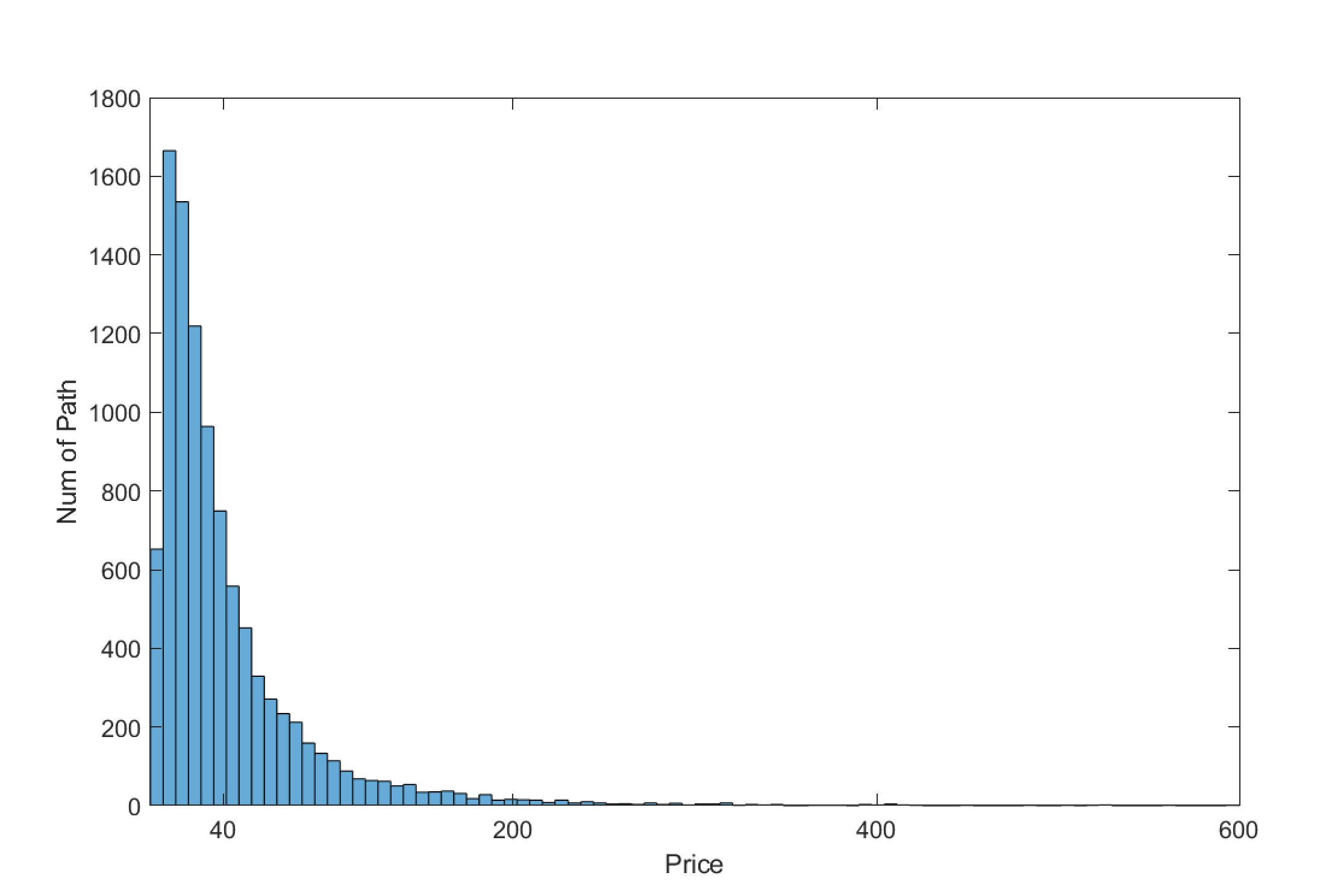}
\caption{\label{F1} Distributions of $S_{T}$ under the CGMY process with default parameters when $S_{0}=40$, generated by Monte Carlo Method.}
\end{figure}



\subsection*{Exposure analyses}
For the Bermudan call option, the trend of EE is to increase first and then decrease, and on the expiry date, the EE will be $0$ since there is no exposure. This trend can be observed in Figure 2 and it presents the EE of examples 1 and 2. The difference between MC-FF method and MC-COS method is quite small, which is also reflected in value adjustments. Figure 3 demonstrates the change of the EE and PFEs of examples 3 and 4, compared with the similar curves of Heston model\cite{del2016efficient}, the EE curve of the CGMY model is more fluctuant due to its pure jump feature.\\
\begin{figure*}
\centering
 \includegraphics[width=0.7\textwidth]{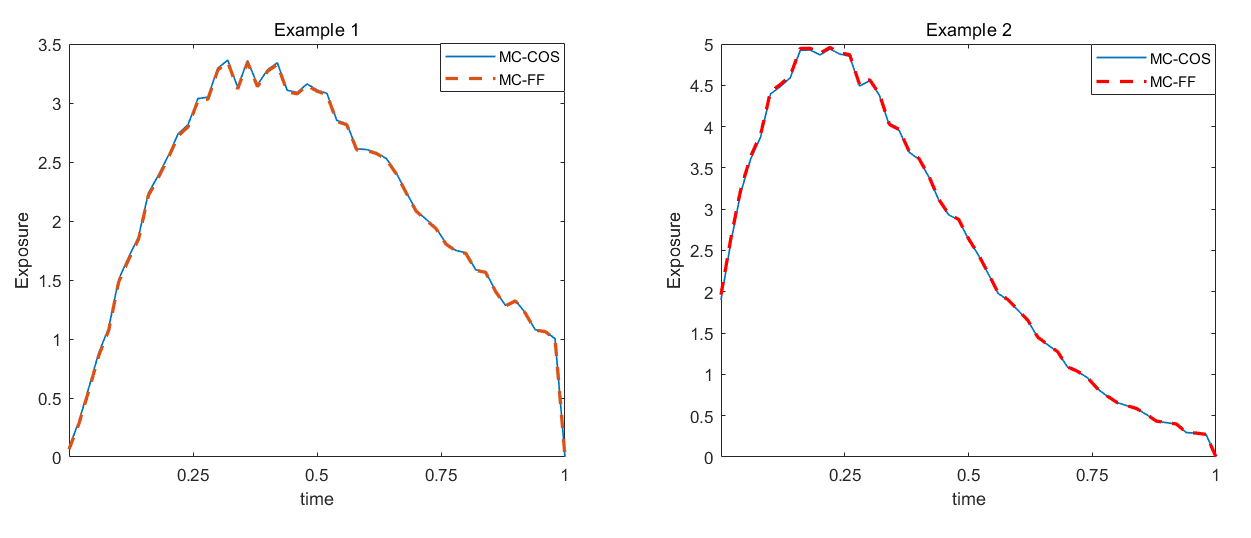}
 \caption{ $EE$ of Examples 1 and 2 for Bermudan call option, comparison of MC-FF method and MC-COS method, for 50 exercise times. }
\end{figure*}
A comparison of examples 1 and 2 in Figure 2 reveals that EE begins at the initial option value and then decreases due to the early exercise probability, but a higher strike price will lead to a more uncertainty. From the plots, we can clearly see that the curve of example 1, which has a higher strike price, is more volatile, while the curve of example 2 is smoother. This result is reasonable. For call options, higher strike prices often mean higher risks. Similarly, the PFEs in Figure 3 also begin with the starting option value because there is no uncertainty at this point. Starting at $t=0$, $PFE_{2.5\%}$ quickly falls to zero, whereas $PFE_{97.5\%}$ is always greater than the EE. Paths will terminate because of the early exercise possibility. That means exercise will occur so that more than 2.5\% of the values are equal to zero shortly. In contrast, the quantity of minimum value for which 97.5\% of the price paths are lower is significantly greater and only reduces later as more and more paths are exercised.

\begin{figure*}
\centering
 \includegraphics[width=0.7\textwidth]{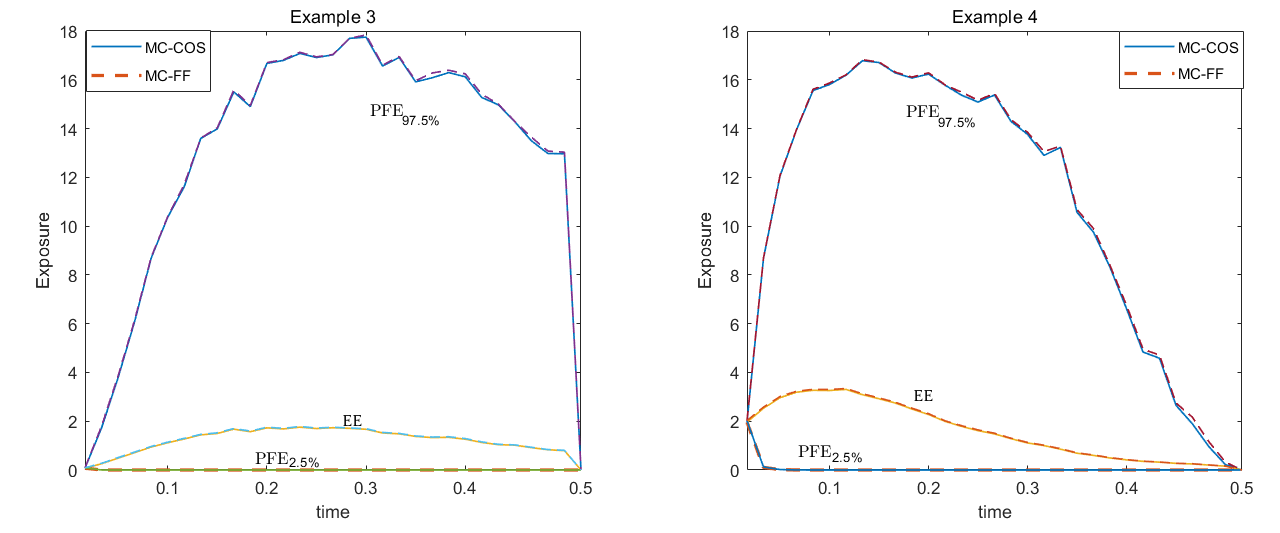}
 \caption{$EE$, $PFE_{97.5\%}$ and $PFE_{2.5\%}$ of Examples 3 and 4 for Bermudan call option, comparison of MC-FF method and MC-COS method, for 12 exercise times. }
\end{figure*}
To investigate the effects of different parameters of the CGMY process on exposure, we conducted four experiments, changing only one parameter each time. In Figure 4, there are clear trends of effects of different parameters. The parameters G and M have almost no effect on the initial value,  but have a greater effect on the high point of the exposure curves. The reason for this feature is that the skewness in the CGMY process is controlled by parameters G and M\cite{cartea2007fractional}. In practice, the selection of parameters for the CGMY process needs to be optimized based on market historical data\cite{bol2008risk}.

\begin{figure*}
\centering
 \includegraphics[width=0.7\textwidth]{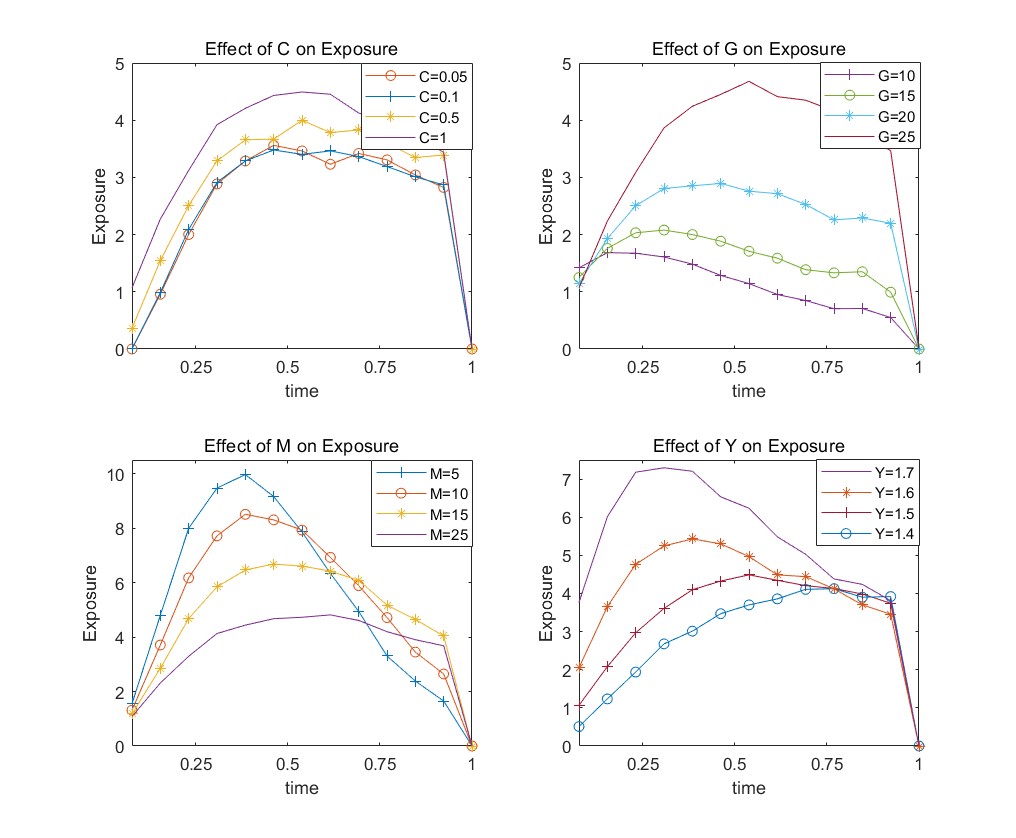}
 \caption{The effects of different parameters in the CGMY model on $EE$ for Bermudan call option, for 12 exercise times. Except for the changing parameters, the other reference parameters are $C=1$, $G=25$, $M=26$, $Y=1.5,$ and strike price $K=50$.}
\end{figure*}

Overall, these results indicate that the trend of EE under the CGMY model is the same as the reality, and the selection of different parameters also has a significant effect on EE. For the calculations of different value adjustments and the comparison of the accuracy and efficiency of the two methods are discussed in the next part.

\subsection*{XVA analyses}
\subsubsection*{Comparison of two methods on XVA}
When we get EE, the value adjustments can be calculated by the definitions in Section II. A. In this work, we take $100$ basis points for the credit spread and $50$ basis points for the funding spread, which are close to the reality values \cite{munro2014monetary}. The total value adjustments of examples 1, 2, 3 and 4 can be found in Table II, and the difference between the two methods is also given. It is apparent from this table that the difference of value adjustments between two methods is quit small. Most of them are smaller than $10^{-3}$. This result suggests that the accuracy of the two methods is very close. As Table II shows, the values of the CVA for both examples 1 and 2 are around 3\%, and the example 2 is slightly larger than example 1 due to the larger exposure. The value of CVA for examples 3 and 4 would be somewhat larger due to their shorter maturities and greater exposures. The overall variation of the FVA is similar to the CVA, although its value is roughly one-third of the CVA, which is also consistent with the theory. In general, the level of XVA is roughly between 4\% and 8\%, which is consistent with the concepts\cite{ruiz2015complete} that value adjustment can be seen as a spread.

Turning to the calculation time in Table III, this data is not similar anymore. It can be seen that the calculation time required by the MC-COS method is very small when there are few simulated paths, but it increases linearly as the number of paths increases. For the MC-FF method, its calculation time does not change significantly with the increase of the simulation path. One reason for this result is that at each exercise time, the MC-COS method performs an additional calculation of the continuation values, which requires interpolation for each path. Although the MC-COS method is more efficient when fewer paths are simulated, it is only accurate enough for Monte Carlo simulation when the number of paths is large enough. In this work, we generated 10,000 Monte Carlo paths for each experiment, and it is clear that the MC-FF method consumes less time.

\begin{table}
\centering
\caption{CVA, FVA and XVA of Examples for the Bermudan call option, comparison and difference of MC-FF method and MC-COS method. }

\begin{tabular}{lllll}
\toprule
	
                     &~~~& MC-FF~~~            & MC-COS~~~&Difference~~~ \\ \cline{2-5}
\multirow{4}{*}{~~~CVA~~~} & Example 1~~ & -3.22\%  &  -3.20\%   &1.70e-04    \\
                     &Example 2~~ & -3.78\%  &-3.77\%  &1.76e-04    \\
                     &Example 3~~&-4.54\% &  -4.47\%  &  7.37e-04 \\
                     &Example 4~~ &  -6.05\%  &-5.97\%& 8.93e-04  \\ \cline{2-5}
\multirow{4}{*}{~~~FVA~~~} & Example 1~~&  -1.08\% &  -1.07\%  &  5.70e-05  \\
                     & Example  2& -1.27\% &    -1.26\% &  5.88e-05 \\
                     & Example 3&  -1.55\% & -1.53\%  & 2.51e-04   \\
                     & Example 4& -2.05\%  &  -2.02\%   & 3.03e-04   \\ \cline{2-5}
\multirow{4}{*}{~~~XVA~~~} & Example 1&  -4.31\%   &-4.28\%  & 2.27e-04   \\
                     &  Example  2&-5.06\%  & -5.04\% &  2.35e-04 \\
                     &  Example 3&-6.09\% & -5.99\%  &  9.89e-04  \\
                     & Example 4&-8.11\% & -8.00\% &  1.19e-03  \\ 	
                     \bottomrule
\end{tabular}

\end{table}
\begin{table}
	\centering
	\caption{Comparison of the calculation time (in seconds) of the two methods with different number of Monte Carlo paths.}

\begin{tabular}{ccccc}
	\toprule
		~&500 Paths&1000 Paths& 5000 Paths &10000 Paths\\
		\cmidrule(r){2-5}
		MC-FF&18.9364&19.3256&21.2162&	22.0326 \\
		MC-COS&3.5263&	6.8362&	36.5698	&76.2355 \\
	\bottomrule
	\end{tabular}

\end{table}

\begin{figure}
	\centering
 \includegraphics[width=0.5\textwidth]{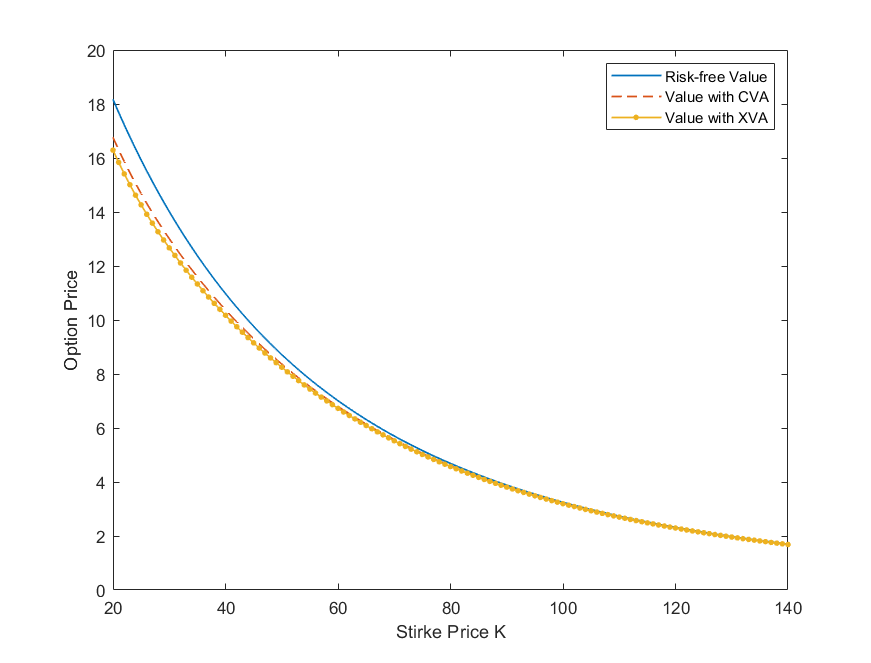}
 \caption{Bermudan call option values with CVA and XVA  against the strike price $K$ with default CGMY parameters.}
\end{figure}
\subsubsection*{Effect of XVA on option value}
In Figure 5, we compare the Bermudan call option value considering CVA and XVA to the risk-free value, which is  the option value without counterparty default risk. As we talked before, the CVA accounts for a large proportion of XVA, and both of them will tend to 0 as $K$ increases. From the Figure 5 we can see that the existence of credit risk and funding risk reduces the value of option. This is intuitive because there is counterparty default risk, which certainly makes the profits obtained through the option smaller for the institutes.

%

\section*{Conclusion}
In this work, we have used two different approaches to find the total value adjustments of the Bermudan option, whose underlying asset follows a complex pure jump CGMY process. Both approaches are based on the Monte Carlo simulation of the CGMY process, and the exposures are calculated at each path and exercise times. Although in this work we only discuss the CGMY model, the approach is similar for other pure jump L\'evy process. For example, when the CGMY parameter $Y=0$, we can obtain the VG model, and the KoBoL model can also be obtained by making appropriate adjustments to the parameters.\\
To obtain the exposure, we applied the finite difference to a FPDE under the CGMY process and proved the convergence of this method (MC-FF method). Compared to the benchmark method (MC-COS method), the MC-FF method we constructed is more efficient when using more Monte Carlo paths. At the same time, the accuracy of the MC-FF method is almost the same as that of the MC-COS method. According to the numerical results, it can be seen that the impact of the pure jump feature on exposure is most significant for $PFE_{97.5\%}$. And CVA, as a part of XVA, has a more obvious effect on the value of the Bermudan option than FVA. \\
Further research should be undertake to explore how to apply realistic market data to find XVA of derivatives and compare it with actual market price. In addition, inspired by Salvador\cite{salvador2020total}, it is also a considerable attempt to derive a differential equation modeling XVA when pure jump feature is assumed.
\section*{Acknowledgements}
This work is supported by University of Macau (MYRG2019-00009-FST and MYRG2018-00025-FST), NSFC (12001067), and Macau Young Scholars Program (AM2020016).

\end{document}